\documentclass[a4paper,11pt,reqno]{amsart}
\usepackage{amsmath}
\usepackage{amstext}
\usepackage{amsbsy}
\usepackage{amsopn}
\usepackage{upref}
\usepackage{amsthm}
\usepackage{amsfonts}
\usepackage{amssymb}
\usepackage{mathrsfs}
\usepackage{times}
\usepackage{cite}
\usepackage{bm}
\allowdisplaybreaks
 \newtheorem{theorem}{Theorem}

 \newtheorem{lemma}[theorem]{Lemma}
\theoremstyle{definition}
 \newtheorem{definition}[theorem]{Definition}
\theoremstyle{remark}
 
\begin{document}
\title[Radon transform]{Inversion of higher dimensional Radon transforms of seismic-type}
\author[H.~Chihara]{Hiroyuki Chihara}
\address{College of Education, University of the Ryukyus, Nishihara, Okinawa 903-0213, Japan}
\email{aji@rencho.me}
\thanks{Supported by the JSPS Grant-in-Aid for Scientific Research \#19K03569.}
\begin{abstract}
We study integral transforms mapping a function on the Euclidean space 
to the family of its integration on some hypersurfaces, 
that is, a function of hypersurfaces. 
The hypersurfaces are given by the graphs of functions 
with fixed axes of the independent variables, 
and are imposed some symmetry with respect to the axes.  
These transforms are higher dimensional version of generalization of 
the parabolic Radon transform and 
the hyperbolic Radon transform arising from seismology. 
We prove the inversion formulas for these transforms 
under some vanishing and symmetry conditions of functions. 
\end{abstract}
\maketitle
\section{Introduction}
\label{section:introduction}
Let $n$ be a positive integer. 
Fix arbitrary 
$c=(c_1,\dotsc,c_n)\in\mathbb{R}^n$ and 
$\alpha_1,\dotsc,\alpha_n,\beta>1$. 
Set $\alpha=(\alpha_1,\dotsc,\alpha_n)$ for short.  
Let 
$$
(x,y)=(x_1,\dotsc,x_n,y),\  
(s,u)=(s_1,\dotsc,s_n,u) 
\in 
\mathbb{R}^{n+1}
=
\mathbb{R}^n\times\mathbb{R}
$$ 
be independent variables of functions. 
We study the inversions of the integral transforms 
$\mathcal{P}_{\alpha}f(s,u)$,  
$\mathcal{Q}_{\alpha}f(s,u)$ 
and  
$\mathcal{R}_{\alpha,\beta}f(s,u)$ of a function $f(x,y)$. 
These are the integrations of $f(x,y)$ 
on some special families of hypersurfaces. 
We do not deal with the general hypersurfaces. 
The precise definitions of our transforms are the following. 
\par
Firstly, $\mathcal{P}_{\alpha}f(s,u)$ is defined by 
\begin{align*}
  \mathcal{P}_{\alpha}f(s,u)
& =
  \int_{\mathbb{R}^n}
  f
  \left(
  x,
  \sum_{i=1}^n
  s_i
  \lvert{x_i-c_i}\rvert^{\alpha_i}
  +
  u
  \right)
  dx
\intertext{which is}
  \mathcal{P}_{\alpha}f(s,u)
& =
  \int_{\mathbb{R}^n}
  f
  \left(
  x+c,
  \sum_{i=1}^n
  s_i
  \lvert{x_i}\rvert^{\alpha_i}
  +
  u
  \right)
  dx. 
\end{align*}
$\mathcal{P}_{\alpha}f(s,u)$ is the integration of $f$ on a hypersurface 
$$
\Gamma_\mathcal{P}(\alpha;s,u)
=
\left\{
(x,y)\in\mathbb{R}^{n+1}
\ : \ 
y
=
\sum_{i=1}^n
s_i
\lvert{x_i-c_i}\rvert^{\alpha_i}
+
u
\right\}, 
$$
and $dx$ is not the standard volume of $\Gamma_\mathcal{P}(\alpha;s,u)$ 
induced by the Euclidean metric of $\mathbb{R}^{n+1}$. 
$\Gamma_\mathcal{P}(2,\dotsc,2;s,u)$ is a paraboloid if $s_i\ne0$ for all $i=1,\dotsc,n$. 
In particular, when $n=1$, 
$\mathcal{P}_2f(s,u)$ is an integration over a parabola in $\mathbb{R}^2$, 
and is called the parabolic Radon transform of $f$ in seismology. 
We assume that 
\begin{align}
& f(x_1,\dotsc,x_{i-1},-x_i+c_i,x_{i+1},\dotsc,x_n,y)
\nonumber
\\
  =
& f(x_1,\dotsc,x_{i-1},x_i+c_i,x_{i+1},\dotsc,x_n,y)
\label{equation:symmetry1}
\end{align}
for $i=1,\dotsc,n$, 
that is, 
$f(-x+c,y)=f(x+c,y)$. 
If we split $f(x)$ into the even and odd parts 
in some $x_i$ with respect to the hyperplane $x_i=c_i$ in $\mathbb{R}^n$, 
then the contribution of the odd part to $\mathcal{P}_\alpha{f}$ becomes $0$, 
and the injectivity of $\mathcal{P}_\alpha$ fails to hold. 
\par
Secondly, $\mathcal{Q}_{\alpha}f(s,u)$ is defined by 
\begin{align*}
  \mathcal{Q}_{\alpha}f(s,u)
&  =
  \int_{\mathbb{R}^n}
  f
  \left(
  x,
  \sum_{i=1}^n
  s_i
  (x_i-c_i)\lvert{x_i-c_i}\rvert^{\alpha_i-1}
  +
  u
  \right)
  dx
\intertext{which is}
  \mathcal{Q}_{\alpha}f(s,u)
& =
  \int_{\mathbb{R}^n}
  f
  \left(
  x+c,
  \sum_{i=1}^n
  s_i
  x_i\lvert{x_i}\rvert^{\alpha_i-1}
  +
  u
  \right)
  dx.
\end{align*}
$\mathcal{Q}_{\alpha}f(s,u)$ is the integration of $f$ on a hypersurface 
$$
\Gamma_\mathcal{Q}(\alpha;s,u)
=
\left\{
(x,y)\in\mathbb{R}^{n+1}
\ : \ 
y
=
\sum_{i=1}^n
s_i
(x_i-c_i)\lvert{x_i-c_i}\rvert^{\alpha_i-1}
+
u
\right\} 
$$
for $f$ and the measure is not the standard volume of 
$\Gamma_\mathcal{Q}(\alpha;s,u)$ 
induced by the Euclidean metric of $\mathbb{R}^{n+1}$. 
We do not need some symmetry conditions like 
\eqref{equation:symmetry1} for $\mathcal{Q}_\alpha$. 
Indeed there is no specific relationship between 
the values of $y=\sum s_ix_i\lvert{x_i}\rvert^{\alpha_i-1}+u$ 
for $x_i$ and $-x_i$ with some $i=1,\dotsc,n$.   
\par
Thirdly, $\mathcal{R}_{\alpha,\beta}f(s,u)$ is defined by 
\begin{align*}
&  \mathcal{R}_{\alpha,\beta}f(s,u)
\\
  =
& \int_{\sum s_i\lvert{x_i-c_i}\rvert^{\alpha_i}+u>0}
  f
  \left(
  x,
  \left(
  \sum_{i=1}^n
  s_i
  \lvert{x_i-c_i}\rvert^{\alpha_i}
  +
  u
  \right)^{1/\beta}
  \right)
  \bigg/
  \left(
  \sum_{i=1}^n
  s_i
  \lvert{x_i-c_i}\rvert^{\alpha_i}
  +
  u
  \right)^{1/\beta}
  dx
\intertext{which is}
&  \mathcal{R}_{\alpha,\beta}f(s,u)
\\
  =
& \int_{\sum s_i\lvert{x_i}\rvert^{\alpha_i}+u>0}
  f
  \left(
  x+c,
  \left(
  \sum_{i=1}^n
  s_i
  \lvert{x_i}\rvert^{\alpha_i}
  +
  u
  \right)^{1/\beta}
  \right)
  \bigg/
  \left(
  \sum_{i=1}^n
  s_i
  \lvert{x_i}\rvert^{\alpha_i}
  +
  u
  \right)^{1/\beta}
  dx.
\end{align*}
$\mathcal{R}_{\alpha,\beta}f(s,u)$ is the integration of $f$ on a hypersurface 
$$
\Gamma_\mathcal{R}(\alpha,\beta;s,u)
=
\left\{
(x,y)\in\mathbb{R}^{n+1}
\ : \ 
\lvert{y}\rvert^\beta
=
\sum_{i=1}^n
s_i
\lvert{x_i-c_i}\rvert^{\alpha_i}
+
u
\right\} 
$$
for $f$ satisfying \eqref{equation:symmetry1} and 
\begin{equation}
f(x+c,-y)=f(x+c,y). 
\label{equation:symmetry2} 
\end{equation}
The condition \eqref{equation:symmetry1} 
is required for the injectivity of 
$\mathcal{R}_{\alpha,\beta}$ 
in the same way as $\mathcal{P}_\alpha$. 
We need the condition \eqref{equation:symmetry2} to obtain the inversion. 
Our method of the proof of the inversion is 
the reduction to the standard Radon transform on $\mathbb{R}^{n+1}$, 
and we can only deal with functions $f(x,y)$ 
defined on the whole space $\mathbb{R}^{n+1}$.   
Note that the measure is not the standard volume of 
$\Gamma_\mathcal{R}(\alpha,\beta;s,u)$ 
induced by the Euclidean metric of $\mathbb{R}^{n+1}$. 
To resolve the singularity at 
$\sum s_i\lvert{x_i-c_i}\rvert^{\alpha_i}+u=0$, 
it is natural to assume that $f(x,0)=0$. 
$\Gamma_\mathcal{R}(2,\dotsc,2,2;s,u)$ 
is a quadratic hypersurface like hyperboloids if 
$s_1,\dotsc,s_n\ne0$. 
In particular, when $n=1$, 
$\mathcal{R}_{2,2}f(s,u)$ is called the hyperbolic Radon transform of $f$ in seismology. 
\par
Here we recall the background of our transforms. 
We start with the transform on $\mathbb{R}^2$, 
that is, the integration on the plane curves. 
In the early 1980s, 
Cormack introduced the Radon transform of a family of plane curves and 
studied the basic properties in his pioneering works 
\cite{Cormack1981} and \cite{Cormack1982}. 
More than a decade later, Denecker, van Overloop and Sommen 
in \cite{Denecker1998} studied the parabolic Radon transform 
without fixed axis, in particular, the support theorem, 
higher dimensional generalization and etc. 
In 2011 which is more than ten years later, 
Jollivet, Nguyen and Truong in \cite{Jollivet2011} 
studied some properties of the parabolic Radon transform with fixed axis, 
which is the exact contour integration. 
Recently, Moon established the inversion of 
the parabolic Radon transform $\mathcal{P}_2$ 
and the inversion of the hyperbolic Radon transform 
$\mathcal{R}_{2,2}$ respectively in his interesting paper \cite{Moon2016b}. 
He introduced some change of variables in $(x,y)\in\mathbb{R}^2$ 
so that the Radon transform of a family of plane curves became 
so-called the X-ray transform, that is, the Radon transform of a family of lines. 
More recently, replacing $x^2$ by some function $\varphi(x)$ 
in the parabolic Radon transform, 
Ustaoglu developed Moon's idea 
to try to obtain the inversion of more general Radon transforms 
on the plane in \cite{Ustaoglu2017}. 
More recently, the author studied 
$\mathcal{P}_\alpha$, $\mathcal{Q}_\alpha$ and $\mathcal{R}_{\alpha,\beta}$ 
on $\mathbb{R}^2$, 
and obtained the inversion formulas for them in \cite{Chihara2019}. 
Those are the mathematical background of our transforms on the plane. 
For the scientific background of 
the parabolic Radon transform $\mathcal{P}_2$ 
and the hyperbolic transform $\mathcal{R}_{2,2}$, 
see 
\cite{Hampson1986} and the introduction of \cite{Moon2016b} 
for both of them, 
\cite{Maeland1998}, \cite{Maeland2000} 
for the parabolic Radon transform,  
and  
\cite{Bickel2000} 
for the hyperbolic Radon transform respectively. 
\par
Here we turn to the higher dimensional case $n\geqq2$. 
There are no mathematical results on our transforms so far. 
Unfortunately, however, 
our transforms with $n\geqq2$ have no scientific background at this time. 
Here we quote some works on the transforms which are integrations on hypersurfaces. 
Moon studied integral transforms over ellipsoids 
in \cite{Moon2014} and \cite{Moon2016a}.  
This arises in synthetic aperture radar (SAR), 
ultrasound reflection tomography (URT) and radio tomography. 
Recently, transforms on cones have been intensively studied. 
This models Compton cameras, 
and is sometimes called cone transform or Compton transform.  
See, e.g., \cite{Kuchment2016} and references therein. 
\par
The aim of the present paper is 
to establish the inversion formulas for 
$\mathcal{P}_\alpha$, $\mathcal{Q}_\alpha$ and $\mathcal{R}_{\alpha,\beta}$. 
Here we introduce some function spaces to state our results. 
These function spaces consists of Schwartz functions on $\mathbb{R}^{n+1}$ 
satisfying some symmetries and vanishing conditions. 
The symmetries are natural for our transforms, 
and the vanishing conditions are used 
for justifying the change of variables 
for the reduction to the standard transform. 
We denote the set of all Schwartz functions on $\mathbb{R}^{n+1}$ by 
$\mathscr{S}(\mathbb{R}^{n+1})$.
\begin{definition}
\label{definition:functionspaces} 
Fix $c=(c_1,\dotsc,c_n)\in\mathbb{R}^n$. 
Let $m=(m_1,\dotsc,m_n)$ be a multi-index of nonnegative integers. 
We define function spaces 
$\mathscr{S}_{c,m}(\mathbb{R}^{n+1})$, 
$\mathscr{S}_{c,m}^\mathcal{P}(\mathbb{R}^{n+1})$ 
and  
$\mathscr{S}_{c,m}^\mathcal{R}(\mathbb{R}^{n+1})$ 
as follows. 
\begin{itemize}
\item[(i)] 
$\mathscr{S}_{c,m}(\mathbb{R}^{n+1})$ 
is the set of all $f(x,y)\in\mathscr{S}(\mathbb{R}^{n+1})$ 
satisfying the following vanishing conditions
$$
\frac{\partial^k f}{\partial x_i^k}(x+c,y)\bigg\vert_{x_i=0}=0,
\quad
k=0,1,\dotsc,m_i,\  
i=1,\dotsc,n.
$$
\item[(ii)] 
$\mathscr{S}_{c,m}^\mathcal{P}(\mathbb{R}^{n+1})$ 
is the set of all $f(x,y)\in\mathscr{S}_{c,m}(\mathbb{R}^{n+1})$ 
satisfying the symmetry \eqref{equation:symmetry1}.  
\item[(iii)]  
$\mathscr{S}_{c,m}^\mathcal{R}(\mathbb{R}^{n+1})$ 
is the set of all $f(x,y)\in\mathscr{S}_{c,m}^\mathcal{P}(\mathbb{R}^{n+1})$ 
satisfying the symmetry \eqref{equation:symmetry2} and the vanishing condition 
$f(x+c,0)=0$.  
\end{itemize}
\end{definition}
Recall $\alpha_1,\dotsc,\alpha_n,\beta>1$. 
Throughout the present paper we assume that the vanishing order 
$m_i$ at $x_i=c_i$ satisfies $m_i \geqq \alpha_i-2$ for all $i=1,\dotsc,n$. 
This condition guarantees the reduction to the standard Radon transform.  
Our main results are the following. 
\begin{theorem}
\label{theorem:main} 
Let $c=(c_1,\dotsc,c_n)\in\mathbb{R}^n$, 
and let $\alpha_1,\dotsc,\alpha_n,\beta>1$. 
Suppose that $m_i$ is a nonnegative integer satisfying 
$m_i \geqq \alpha_i-2$ for all $i=1,\dotsc,n$. 
\begin{itemize}
\item[{\rm (i)}] 
For any $f\in\mathscr{S}_{c,m}^\mathcal{P}(\mathbb{R}^{n+1})$, 
if $n$ is odd, then 
\begin{align}
  f(x,y)
& =
  \frac{4\cdot(-1)^{(n-1)/2}}{(4\pi)^{n+1}}
  \left(
  \prod_{i=1}^n
  \alpha_i
  \lvert{x_i-c_i}\rvert^{\alpha_i-1}
  \right)
\nonumber
\\
& \times
  \int_{\mathbb{R}^n}
  \left(
  \operatorname{pv}
  \int_{-\infty}^\infty
  \frac{\partial_u^n \mathcal{P}_{\alpha}f(s,u)}{y-\sum_{i=1}^ns_i\lvert{x_i-c_i}\rvert^{\alpha_i}-u}
  du
  \right)
  ds,
\label{equation:inversion1} 
\end{align}
and if $n$ is even, then 
\begin{align}
  f(x,y)
& =
  \frac{(-1)^{n/2}}{(4\pi)^n}
  \left(
  \prod_{i=1}^n
  \alpha_i
  \lvert{x_i-c_i}\rvert^{\alpha_i-1}
  \right)
\nonumber
\\
& \times
  \int_{\mathbb{R}^n}
  (\partial_u^n \mathcal{P}_{\alpha}f)
  \left(
  s,
  y-\sum_{i=1}^ns_i\lvert{x_i-c_i}\rvert^{\alpha_i}
  \right)
  ds.
\label{equation:inversion2} 
\end{align}
\item[{\rm (ii)}] 
For any $f\in\mathscr{S}_{c,m}(\mathbb{R}^{n+1})$, 
if $n$ is odd, then 
\begin{align}
  f(x,y)
& =
  \frac{2\cdot(-1)^{(n-1)/2}}{(2\pi)^{n+1}}
  \left(
  \prod_{i=1}^n
  \alpha_i
  \lvert{x_i-c_i}\rvert^{\alpha_i-1}
  \right)
\nonumber
\\
& \times
  \int_{\mathbb{R}^n}
  \left(
  \operatorname{pv}
  \int_{-\infty}^\infty
  \frac{\partial_u^n \mathcal{Q}_{\alpha}f(s,u)}{y-\sum_{i=1}^ns_i(x_i-c_i)\lvert{x_i-c_i}\rvert^{\alpha_i-1}-u}
  du
  \right)
  ds,
\label{equation:inversion5} 
\end{align}
and if $n$ is even, then 
\begin{align}
  f(x,y)
& =
  \frac{(-1)^{n/2}}{(2\pi)^n}
  \left(
  \prod_{i=1}^n
  \alpha_i
  \lvert{x_i-c_i}\rvert^{\alpha_i-1}
  \right)
\nonumber
\\
& \times
  \int_{\mathbb{R}^n}
  (\partial_u^n \mathcal{Q}_{\alpha}f)
  \left(
  s,
  y-\sum_{i=1}^ns_i(x_i-c_i)\lvert{x_i-c_i}\rvert^{\alpha_i-1}
  \right)
  ds.
\label{equation:inversion6} 
\end{align}
\item[{\rm (iii)}] 
For any $f\in\mathscr{S}_{c,m}^\mathcal{R}(\mathbb{R}^{n+1})$, 
if $n$ is odd, then 
\begin{align}
  f(x,y)
& =
  \frac{4\cdot(-1)^{(n-1)/2}}{(4\pi)^{n+1}}
  \left(
  \prod_{i=1}^n
  \alpha_i
  \lvert{x_i-c_i}\rvert^{\alpha_i-1}
  \right)
  \lvert{y}\rvert
\nonumber
\\
& \times
  \int_{\mathbb{R}^n}
  \left(
  \operatorname{pv}
  \int_{-\infty}^\infty
  \frac{\partial_u^n \mathcal{R}_{\alpha,\beta}f(s,u)}{\lvert{y}\rvert^\beta-\sum_{i=1}^ns_i\lvert{x_i-c_i}\rvert^{\alpha_i}-u}
  du
  \right)
  ds,
\label{equation:inversion3} 
\end{align}
and if $n$ is even, then 
\begin{align}
  f(x,y)
& =
  \frac{(-1)^{n/2}}{(4\pi)^n}
  \left(
  \prod_{i=1}^n
  \alpha_i
  \lvert{x_i-c_i}\rvert^{\alpha_i-1}
  \right)
  \lvert{y}\rvert
\nonumber
\\
& \times
  \int_{\mathbb{R}^n}
  (\partial_u^n \mathcal{R}_{\alpha,\beta}f)
  \left(
  s,
  \lvert{y}\rvert^\beta-\sum_{i=1}^ns_i\lvert{x_i-c_i}\rvert^{\alpha_i}
  \right)
  ds.
\label{equation:inversion4} 
\end{align}
\end{itemize}
\end{theorem}
We develop the method in \cite{Chihara2019} and prove Theorem~\ref{theorem:main}. 
The basic idea is the reduction to the standard Radon transform due to Moon in \cite {Moon2016b}. 
We begin with the basic facts on the standard Radon transform in Section~\ref{section:radon}. 
Section~\ref{section:coordinatization} is devoted to studying the coordinatization 
of the upper hemisphere by $s\in\mathbb{R}^n$. 
Next, we prepare some lemmas related with vanishing conditions in Section~\ref{section:vanishing}. 
Finally, we prove Theorem~\ref{theorem:main} in Section~\ref{section:proofs}.  
\section{The standard Radon transform}
\label{section:radon}
We recall the definition of the standard Radon transform on $\mathbb{R}^{n+1}$ 
and the inversion formula. 
Let $\mathbb{S}^n$ be the unit sphere in $\mathbb{R}^n$ defined by 
$$
\mathbb{S}^n
=
\{
\omega=(\omega_1,\dotsc,\omega_{n+1})\in\mathbb{R}^{n+1}
\ : \ 
\lvert\omega\rvert^2
=
\omega_1^2+\dotsb+\omega_{n+1}^2=1
\}, 
$$
and let $\mathbb{S}^n_+$ be the upper hemisphere 
$$
\mathbb{S}^n_+
=
\{
\omega=(\omega_1,\dotsc,\omega_{n+1})\in\mathbb{S}^n
\ : \ 
\omega_{n+1}>0
\}.
$$
Set 
$$
H(\omega,t)
=
\{
\zeta\in\mathbb{R}^{n+1}
\ : \
\langle{\zeta,\omega}\rangle=t
\},
\quad
\omega\in\mathbb{S}^n,\ 
t\in\mathbb{R},
$$
where $\langle\cdot,\cdot\rangle$ is the standard inner product of the Euclidean space. 
Note that $H(-\omega,-t)=H(\omega,t)$, and 
$$
H(\omega,t)
=
\{\zeta+t\omega \ : \ \zeta \in H(\omega,0)\} 
$$
since 
$\langle{\zeta+t\omega,\omega}\rangle=t$ for $\zeta \in H(\omega,0)$. 
$H(\omega,t)$ is a hyperplane in $\mathbb{R}^{n+1}$ 
which is perpendicular to $\omega$ and is passing through $t\omega$. 
In particular, $H(\omega,0)$ is the orthogonal complement of $\{\omega\}$ in $\mathbb{R}^{n+1}$. 
The Radon transform of a function $F(\zeta)$ is defined by 
$$
\mathcal{X}F(\omega,t)
=
\int_{H(\omega,t)}
F(\zeta)
dm(\zeta)
=
\int_{H(\omega,0)} 
F(t\omega+\zeta)
dm(\zeta),
$$
where $dm$ is the induced measure on the hyperplane 
from the Lebesgue measure on $\mathbb{R}^{n+1}$. 
Note that $\mathcal{X}F(-\omega,-t)=\mathcal{X}F(\omega,t)$. 
In the present paper, we deal with integrations on the graph of 
a function $\langle{s,\xi}\rangle+u$ of the variables $\xi\in\mathbb{R}^n$ 
with some constants $(s,u)\in\mathbb{R}^n\times\mathbb{R}$. 
It is easy to see that
\begin{equation}
\int_{\mathbb{R}^n}
F(\xi,\langle{s,\xi}\rangle+u)
d\xi
=
\frac{1}{\sqrt{1+\lvert{s}\rvert^2}}
\mathcal{X}F
\left(
\frac{(-s,1)}{\sqrt{1+\lvert{s}\rvert^2}},
\frac{u}{\sqrt{1+\lvert{s}\rvert^2}}
\right)
\label{equation:radonF} 
\end{equation} 
since 
$$
\{
(\xi,\langle{s,\xi}\rangle+u)
\ : \ 
\xi\in\mathbb{R}^n
\}
=
H
\left(
\frac{(-s,1)}{\sqrt{1+\lvert{s}\rvert^2}},
\frac{u}{\sqrt{1+\lvert{s}\rvert^2}}
\right), 
\quad
dm=\sqrt{1+\lvert{s}\rvert^2}d\xi. 
$$
The inversion formula is as follows. 
\begin{theorem}
\label{theorem:inversionF} 
For $F(\xi,\eta) \in \mathscr{S}(\mathbb{R}^{n+1})$, 
if $n$ is odd, then 
$$
F(\xi,\eta)
=
\frac{2\cdot(-1)^{(n-1)/2}}{(2\pi)^{n+1}}
\int_{\mathbb{S}^n_+}
\left(
\operatorname{pv}
\int_{-\infty}^\infty
\frac{\partial_t^n \mathcal{X}F(\omega,t)}{\langle(\xi,\eta),\omega\rangle-t}
dt
\right)
d\omega,
$$
and if $n$ is even, then 
$$
F(\xi,\eta)
=
\frac{(-1)^{n/2}}{(2\pi)^n}
\int_{\mathbb{S}^n_+}
(\partial_t^n\mathcal{X}F)
(\omega,\langle(\xi,\eta),\omega\rangle)
d\omega,
$$
where $d\omega$ is the induced measure on $\mathbb{S}^n$ from $\mathbb{R}^{n+1}$. 
\end{theorem}
For Theorem~\ref{theorem:inversionF}, see, e.g., 
Corollary~2.6 and the Remark below in page 33 of Palamodov's textbook \cite{Palamodov2004}. 
It is important to mention that Theorem~\ref{theorem:inversionF} holds for 
smooth functions $F(\xi,\eta)$ satisfying 
$F(\xi,\eta)=O\bigl((1+\lvert\xi\rvert+\lvert\eta\rvert)\bigr)^{-d}$ 
with some $d>n$, 
compactly supported distributions, 
rapidly decaying Lebesgue measurable functions, and etc. 
See, e.g., \cite{Helgason2011} for the detail. 
  
\section{Coordinatization of hemisphere}
\label{section:coordinatization}
We introduce a coordinatization $s\in\mathbb{R}^n$ 
of the upper hemisphere $\mathbb{S}^n_+$. 
The polar coordinates 
$(\theta_1,\dotsc,\theta_n)\in(0,\pi)^n$ 
for the point $\omega\in\mathbb{S}^n_+$ is given by 
$$
\omega_1=\cos\theta_1, 
\quad
\omega_i=\sin\theta_1\dotsb\sin\theta_{i-1}\cdot\cos\theta_i
\quad
(i=1,\dotsc,n), 
\quad
\omega_{n+1}=\sin\theta_1\dotsb\sin\theta_n, 
$$
and the volume form $d\omega$ given by 
$$
d\omega
=
d\theta_1 
\quad
(n=1), 
\quad
d\omega
=
\left(
\prod_{i=1}^{n-1}
\sin^{n-i}\theta_i
\right)
d\theta_1 \dotsb d\theta_n
\quad
(n=2,3,4,\dotsc)
$$
is well-known. 
\par
Since $\omega_1^2+\dotsb+\omega_{n+1}^2=1$, 
we introduce new coordinates 
$s=(s_1,\dotsc,s_n)\in\mathbb{R}^n$ of $\mathbb{S}^n_+$ 
defined by 
\begin{equation}
\omega
=
\frac{(-s,1)}{\sqrt{1+\lvert{s}\rvert^2}},
\quad\text{i.e.,}\quad
(\omega_1,\dotsc,\omega_n,\omega_{n+1})
=
\frac{(-s_1,\dotsc,-s_n,1)}{\sqrt{1+\lvert{s}\rvert^2}}.
\label{equation:change222} 
\end{equation}
Note that $s$ moves in $\mathbb{R}^n$ 
if and only if $\theta$ moves $(0,\pi)^n$. 
Moreover we have 
$$
s_i
=
-
\frac{\cot\theta_i}{\sin\theta_{i+1}\dotsb\sin\theta_n}
\quad
(i=1,\dotsc,n-1), 
\quad
s_n
=
-
\cot\theta_n. 
$$
Elementary calculus yields 
$$
\frac{\partial s_i}{\partial \theta_j}
=
0
\quad
(i=2,\dotsc,n,\ j<i),
$$  
$$
\frac{\partial s_i}{\partial \theta_i}
=
\frac{1}{\sin^2\theta_i\cdot\sin\theta_{i+1}\dotsb\sin\theta_n}
\quad
(i=1,\dotsc,n-1), 
\quad
\frac{\partial s_n}{\partial \theta_n}
=
\frac{1}{\sin^2\theta_n}. 
$$
Hence we have
$$
\frac{\partial(s_1,\dotsc,s_n)}{\partial(\theta_1,\dotsc,\theta_n)}
=
\det
\begin{bmatrix}
\dfrac{\partial s_1}{\partial \theta_1} & \ast & \dotsb & \ast
\\
0 & \ddots & \ddots & \vdots
\\
\vdots & \ddots & \ddots & \ast
\\
0 & \dotsb & 0 & \dfrac{\partial s_n}{\partial s_n} 
\end{bmatrix}
=
\prod_{i=1}^n
\frac{\partial s_i}{\partial \theta_i}
=
\left(
\prod_{i=1}^n
\sin^{i+1}\theta_i
\right)^{-1},
$$
and 
\begin{align}
  d\omega
& =
  \left(
  \prod_{j=1}^{n-1}
  \sin^{n-j}\theta_j
  \right)
  \left\lvert
  \frac{\partial(s_1,\dotsc,s_n)}{\partial(\theta_1,\dotsc,\theta_n)}
  \right\rvert^{-1}
  ds
\nonumber
\\
& =
  \left(
  \prod_{j=1}^{n-1}
  \sin^{n-j}\theta_j
  \cdot
  \prod_{i=1}^n
  \sin^{i+1}\theta_i
  \right)
  ds
\nonumber
\\
& =
  \left(
  \prod_{i=1}^n
  \sin\theta_i
  \right)^{n+1}
  ds
  =
  \frac{1}{(1+\lvert{s}\rvert^2)^{(n+1)/2}}\
  ds.
\label{equation:jacobian1} 
\end{align}
We also use change of variables 
\begin{equation}
(\omega,t)
=
\left(
\frac{(-s,1)}{\sqrt{1+\lvert{s}\rvert^2}}, 
\frac{u}{\sqrt{1+\lvert{s}\rvert^2}}
\right). 
\label{equation:change111}
\end{equation}
In this case we have 
\begin{align*}
  \frac{\partial(\theta_1,\dotsc,\theta_n,t)}{\partial(s_1,\dotsc,s_n,u)}
& =
  \det
  \begin{bmatrix}
  \dfrac{\partial \theta_i}{\partial s_j} & \bm{0}
  \\
  \bm{0} & \dfrac{\partial t}{\partial u} 
  \end{bmatrix}
  =
  \frac{\partial(\theta_1,\dotsc,\theta_n)}{\partial(s_1,\dotsc,s_n)}
  \cdot
  \dfrac{\partial t}{\partial u}
\\
& =
  \left(
  \prod_{i=1}^n
  \sin^{i+1}\theta_i
  \right)
  \cdot
  \frac{1}{\sqrt{1+\lvert{s}\rvert^2}}, 
\end{align*}
and 
\begin{equation}
d\omega dt
=     
\frac{1}{(1+\lvert{s}\rvert^2)^{(n+2)/2}}\ dsdu.
\label{equation:jacobian2} 
\end{equation}
%
%
\section{Vanishing conditions}
\label{section:vanishing}
We prepare some lemmas related to vanishing conditions and symmetries. 
These lemmas are used for reducing our transforms  
to the standard Radon transform in the next section. 
We need the following lemma 
to make full use of the vanishing conditions. 
\begin{lemma}
\label{theorem:taylor} 
\quad
\begin{itemize}
\item[{\rm (i)}]
For $f(x,y) \in \mathscr{S}_{c,m}(\mathbb{R}^{n+1})$, 
\begin{align}
& f(x+c,y)
\nonumber
\\
  =
& \frac{x_2^{m_2+1} \dotsb x_n^{m_n+1}}{m_2! \dotsb m_n!}
  \int_0^1 \dotsb \int_0^1
  (1-t_2)^{m_2} \dotsb (1-t_n)^{m_n}
\nonumber
\\
  \times
& \frac{\partial^{m_2+\dotsb+m_n+n-1} f}{\partial x_2^{m_2+1} \dotsb \partial x_n^{m_n+1}}(x_1+c_1,t_2x_2+c_2,\dotsc,t_nx_n+c_n,y)
  dt_2 \dotsb dt_n,
\label{equation:taylor1}
\\
  =
& \frac{x_1^{m_1+1} \dotsb x_n^{m_n+1}}{m_1! \dotsb m_n!}
  \int_0^1 \dotsb \int_0^1
  (1-t_1)^{m_1} \dotsb (1-t_n)^{m_n}
\nonumber
\\
  \times
& \frac{\partial^{m_1+\dotsb+m_n+n} f}{\partial x_1^{m_1+1} \dotsb \partial x_n^{m_n+1}}(t_1x_1+c_1,\dotsc,t_nx_n+c_n,y)
  dt_1 \dotsb dt_n.
\label{equation:taylor2} 
\end{align}
\item[{\rm (ii)}] 
For $f(x,y) \in \mathscr{S}_{c,m}^\mathcal{R}(\mathbb{R}^{n+1})$, 
\begin{align}
& f(x+c,y)
\nonumber
\\
  =
& \frac{x_2^{m_2+1} \dotsb x_n^{m_n+1} y}{m_2! \dotsb m_n!}
  \int_0^1 \dotsb \int_0^1
  (1-t_2)^{m_2} \dotsb (1-t_n)^{m_n}
\nonumber
\\
  \times
& \frac{\partial^{m_2+\dotsb+m_n+n} f}{\partial x_2^{m_2+1} \dotsb \partial x_n^{m_n+1} \partial y}(x_1+c_1,t_2x_2+c_2,\dotsc,t_nx_n+c_n,\tau y)
  dt_2 \dotsb dt_n d\tau,
\label{equation:taylor3}
\\
  =
& \frac{x_1^{m_1+1} \dotsb x_n^{m_n+1} y}{m_1! \dotsb m_n!}
  \int_0^1 \dotsb \int_0^1
  (1-t_1)^{m_1} \dotsb (1-t_n)^{m_n}
\nonumber
\\
  \times
& \frac{\partial^{m_1+\dotsb+m_n+n+1} f}{\partial x_1^{m_1+1} \dotsb \partial x_n^{m_n+1} \partial y}(t_1x_1+c_1,\dotsc,t_nx_n+c_n,\tau y)
  dt_1 \dotsb dt_n d\tau.
\label{equation:taylor4} 
\end{align}
\end{itemize}
We can replace the role of $x_1$ in 
{\rm \eqref{equation:taylor1}} 
and 
{\rm \eqref{equation:taylor3}} 
by the other $x_i$, $i=2,\dotsc,n$. 
\end{lemma}
\begin{proof}
Here we prove \eqref{equation:taylor1} and \eqref{equation:taylor2} for $n=2$. 
The other parts can be proved in the same way. We omit the detail. 
Suppose $f(x_1,x_2,y) \in \mathscr{S}_{c,m}^\mathcal{P}(\mathbb{R}^3)$. 
Since 
$$
\frac{\partial^k f}{\partial x_2^k}(x_1+c_1,c_2,y)=0,
\quad
k=0,1,\dotsc,m_2. 
$$
Taylor's formula gives
\begin{align}
  f(x_1+c_1,x_2+c_2,y)
& =
  \sum_{k=0}^{m_2}
  \frac{x_2^k}{k!}
  \frac{\partial^k f}{\partial x_2^k}(x_1+c_1,c_2,y)
\nonumber
\\
& +
  \frac{x_2^{m_2+1}}{m_2!}
  \int_0^1
  (1-t_2)^{m_2}
  \frac{\partial^{m_2+1} f}{\partial x_2^{m_2+1}}(x_1+c_1,t_2x_2+c_2,y)
  dt_2
\nonumber
\\
& =
  \frac{x_2^{m_2+1}}{m_2!}
  \int_0^1
  (1-t_2)^{m_2}
  \frac{\partial^{m_2+1} f}{\partial x_2^{m_2+1}}(x_1+c_1,t_2x_2+c_2,y)
  dt_2.
\label{equation:taylor101}
\end{align}
This is \eqref{equation:taylor1} for $n=2$. 
Since 
$$
\frac{\partial^k f}{\partial x_1^k}(c_1,x_2+c_2,y)=0,
\quad
k=0,1,\dotsc,m_1
$$
$$
\frac{\partial^{k+l} f}{\partial x_1^k \partial x_2^l}(c_1,x_2+c_2,y)=0,
\quad
k=0,1,\dotsc,m_1, l=0,1,2,3,\dotsc. 
$$
If we use this with $l=m_2+1$, we have 
\begin{align*}
& \frac{\partial^{m_2+1} f}{\partial x_2^{m_2+1}}(x_1+c_1,t_2x_2+c_2,y)
\\
  =
& \frac{x_1^{m_1+1}}{m_1!}
  \int_0^1
  (1-t_1)^{m_1}
  \frac{\partial^{m_1+m_2+2} f}{\partial x_1^{m_1+1} \partial x_2^{m_2+1}}(t_1x_1+c_1,t_2x_2+c_2,y)
  dt_1.  
\end{align*}
Substitute this into \eqref{equation:taylor101}. 
Thus we obtain 
\begin{align*}
  f(x_1+c_1,x_2+c_2,y)
& =
  \frac{x_1^{m_1+1} x_2^{m_2+1}}{m_1! m_2!}
  \int_0^1\int_0^1
  (1-t_1)^{m_1}(1-t_2)^{m_2}
\\
& \times
  \frac{\partial^{m_1+m_2+2} f}{\partial x_1^{m_1+1} \partial x_2^{m_2+1}}(t_1x_1+c_1,t_2x_2+c_2,y)
  dt_1 dt_2,   
\end{align*}
which is \eqref{equation:taylor2} for $n=2$. 
\end{proof} 
Now we introduce functions defined by $f$, 
which is used for reducing our transforms to the standard Radon transform. 
For $f(x,y)$, set 
\begin{align*}
  F_\alpha^\mathcal{P}(\xi,\eta)
& =
  \begin{cases}
  \dfrac{2^n f(\xi_1^{1/\alpha_1}+c_1,\dotsc,\xi_n^{1/\alpha_n}+c_n,\eta)}{\alpha_1\dotsb\alpha_n \cdot \xi_1^{(\alpha_1-1)/\alpha_1}\dotsb\xi_n^{(\alpha_n-1)/\alpha_n}} 
  &\ (\xi_1,\dotsc,\xi_n>0),
  \\
  0 
  &\ (\text{otherwise}),   
  \end{cases} 
\\
  F_\alpha^\mathcal{Q}(\xi,\eta)
& =
  \frac{f(\xi_1\lvert{\xi_1}\rvert^{-1+1/\alpha_1}+c_1,\dotsc,\xi_n\lvert\xi_n\rvert^{-1+1/\alpha_n}+c_n,\eta)}{\alpha_1\dotsb\alpha_n \cdot \xi_1^{(\alpha_1-1)/\alpha_1}\dotsb\xi_n^{(\alpha_n-1)/\alpha_n}} 
  \quad (\xi_1, \dotsc, \xi_n \ne 0),
\intertext{and}
  F_{\alpha,\beta}^\mathcal{R}(\xi,\eta)
& =
  \begin{cases}
  \dfrac{2^n f(\xi_1^{1/\alpha_1}+c_1,\dotsc,\xi_n^{1/\alpha_n}+c_n,\eta^{1/\beta})}{\alpha_1\dotsb\alpha_n \cdot \xi_1^{(\alpha_1-1)/\alpha_1}\dotsb\xi_n^{(\alpha_n-1)/\alpha_n} \eta^{1/\beta}} 
&\ (\xi_1,\dotsc,\xi_n, \eta>0),
\\
0 
&\ (\text{otherwise}).   
  \end{cases} 
\end{align*}
Note that 
$F_\alpha^\mathcal{P}(\xi,\eta)=2^n F_\alpha^\mathcal{Q}(\xi,\eta)$ 
for $\xi_1,\dotsc,\xi_n>0$. 
\begin{lemma}
\label{theorem:change}
Suppose that $\alpha_1 \geqq m_i-2$ for all $i=1,\dotsc,n$.
\begin{itemize}
\item[{\rm (i)}] 
For $f(x,y) \in \mathscr{S}_{c,m}^\mathcal{P}(\mathbb{R}^{n+1})$, 
\begin{equation}
f(x,y)
=
\frac{1}{2^n}
\left(
\prod_{i=1}^n
\alpha_i
\lvert{x_i-c_i}\rvert^{\alpha_i-1}
\right)
F_{\alpha}^\mathcal{P}
(\lvert{x_1-c_1}\rvert^{\alpha_1},\dotsc,\lvert{x_n-c_n}\rvert^{\alpha_n},y),
\label{equation:inverse1} 
\end{equation}
and for any $N>0$, there exists a constant $C_N>0$ such that 
\begin{equation}
\lvert{F_\alpha^\mathcal{P}(\xi,\eta)}\rvert
\leqq
C_N
(1+\lvert\xi\rvert+\lvert\eta\rvert)^{-N}. 
\label{equation:decay1} 
\end{equation}
Moreover, when $\xi\in(0,\infty)^n$ tends to the boundary, 
$F_\alpha^\mathcal{P}(\xi,\langle{s,\xi}\rangle+u)$ 
has a finite limit for any $(s,u)\in\mathbb{R}^n\times\mathbb{R}$.
\item[{\rm (ii)}] 
For $f(x,y) \in \mathscr{S}_{c,m}(\mathbb{R}^{n+1})$, 
\begin{align}
  f(x,y)
& =
  \left(
  \prod_{i=1}^n
  \alpha_i
  \lvert{x_i-c_i}\rvert^{\alpha_i-1}
  \right)
\nonumber
\\
& \times
  F_{\alpha}^\mathcal{Q}
  \bigl(
  (x_1-c_1)\lvert{x_1-c_1}\rvert^{\alpha_1-1},
  \dotsc,
  (x_n-c_n)\lvert{x_n-c_n}\rvert^{\alpha_n-1},
  y
  \bigr),
\label{equation:inverse2} 
\end{align}
and for any $N>0$, there exists a constant $C_N>0$ such that 
\begin{equation}
\lvert{F_\alpha^\mathcal{Q}(\xi,\eta)}\rvert
\leqq
C_N
(1+\lvert\xi\rvert+\lvert\eta\rvert)^{-N}. 
\label{equation:decay2} 
\end{equation}
Moreover, when $\sigma\xi\in(0,\infty)^n$ tends to the boundary, 
$F_\alpha^\mathcal{P}(\xi,\langle{s,\xi}\rangle+u)$ 
has a finite limit for any $(s,u)\in\mathbb{R}^n\times\mathbb{R}$. 
Here $\sigma=(\sigma_1,\dotsc,\sigma_n)\in\{\pm1\}^n$ and 
$\sigma\xi=(\sigma_1\xi_1,\dotsc,\sigma_n\xi_n)$.
\item[{\rm (iii)}] 
For $f(x,y) \in \mathscr{S}_{c,m}^\mathcal{R}(\mathbb{R}^{n+1})$, 
\begin{equation}
f(x,y)
=
\frac{1}{2^n}
\left(
\prod_{i=1}^n
\alpha_i
\lvert{x_i-c_i}\rvert^{\alpha_i-1}
\right) 
\lvert{y}\rvert 
\cdot
F_{\alpha,\beta}^\mathcal{R}
(\lvert{x_1-c_1}\rvert^{\alpha_1},\dotsc,\lvert{x_n-c_n}\rvert^{\alpha_n},\lvert{y}\rvert^\beta),
\label{equation:inverse3} 
\end{equation}
and for any $N>0$, there exists a constant $C_N>0$ such that 
\begin{equation}
\lvert{F_{\alpha,\beta}^\mathcal{R}(\xi,\eta)}\rvert
\leqq
C_N
(1+\lvert\xi\rvert+\lvert\eta\rvert)^{-N}. 
\label{equation:decay3} 
\end{equation}
Moreover, when $\xi\in(0,\infty)^n$ tends to the boundary, 
$F_{\alpha,\beta}^\mathcal{R}(\xi,\langle{s,\xi}\rangle+u)$ 
has a finite limit for any $(s,u)\in\mathbb{R}^n\times\mathbb{R}$.
\end{itemize}
\end{lemma}
\begin{proof}
We prove only (i) here. 
(ii) and (iii) can be proved in the same way as (i). 
We omit the detail. 
\par
Suppose that $f\in\mathscr{S}_{c,m}^\mathcal{P}(\mathbb{R}^{n+1})$. 
Direct computation gives \eqref{equation:inverse1} immediately. 
Set $C_0=2^n/\alpha_1\dotsb\alpha_n$ for short. 
It suffices to consider 
$F_\alpha^\mathcal{P}(\xi,\eta)$ only in $\xi\in(0,\infty)^n$. 
Note that \eqref{equation:taylor2} implies that 
\begin{align}
  F_\alpha^\mathcal{P}(\xi,\eta)
& =
  \frac{C_0}{m_1! \dotsb m_n!}
  \prod_{i=1}^n
  \xi_i^{(m_i+2-\alpha_i)/\alpha_i} 
  \int_0^1\dotsb\int_0^1
  (1-t_1)^{m_1}\dotsb(1-t_n)^{m_n}
\nonumber
\\
& \times 
  \frac{\partial^{m_1+\dotsb+m_n+n} f}{\partial x_1^{m_1+1} \dotsb \partial x_n^{m_n+1}}
  (t_1\xi_1^{1/\alpha_1}+c_1,\dotsc,t_n\xi_n^{1/\alpha_n}+c_n,\eta)
  dt_1 \dotsb dt_n.
\label{equation:taylor201}
\end{align}
This shows that 
$F_\alpha^\mathcal{P}(\xi,\eta)$ 
is bounded for 
$\lvert\xi\rvert+\lvert\eta\rvert<1$ 
since 
$m_i+2-\alpha_i\geqq0$ for $i=1,\dotsc,n$.  
When $\lvert\xi\rvert+\lvert\eta\rvert\geqq1$, 
it follow that 
$\lvert\eta\rvert\geqq(\lvert\xi\rvert+\lvert\eta\rvert)/(\sqrt{n}+1)$ 
or 
$\xi_i\geqq(\lvert\xi\rvert+\lvert\eta\rvert)/(\sqrt{n}+1)$ 
for some $i=1,\dotsc,n$, say $i=1$. 
When 
$\lvert\eta\rvert\geqq(\lvert\xi\rvert+\lvert\eta\rvert)/(\sqrt{n}+1)\geqq1/(\sqrt{n}+1)$, 
\eqref{equation:taylor201} shows \eqref{equation:decay1}. 
Note that \eqref{equation:taylor1} implies that 
\begin{align}
  F_\alpha^\mathcal{P}(\xi,\eta)
& =
  \frac{C_0}{\xi_1^{(\alpha_1-1)/\alpha_1} m_2! \dotsb m_n!}
  \prod_{i=2}^n
  \xi_i^{(m_i+2-\alpha_i)/\alpha_i} 
  \int_0^1\dotsb\int_0^1
  (1-t_2)^{m_2}\dotsb(1-t_n)^{m_n}
\nonumber
\\
& \times 
  \frac{\partial^{m_2+\dotsb+m_n+n-1} f}{\partial x_2^{m_2+1} \dotsb \partial x_n^{m_n+1}}
  (\xi_1^{1/\alpha_1}+c_1,t_2\xi_2^{1/\alpha_2}+c_2,\dotsc,t_n\xi_n^{1/\alpha_n}+c_n,\eta)
  dt_2 \dotsb dt_n.
\label{equation:taylor202}
\end{align}
When 
$\xi_1\geqq(\lvert\xi\rvert+\lvert\eta\rvert)/(\sqrt{n}+1)\geqq1/(\sqrt{n}+1)$, 
\eqref{equation:taylor202} shows \eqref{equation:decay1}. 
Combining the above all, we prove \eqref{equation:decay1}. 
We can prove the existence and the finiteness of 
the limit of $F_\alpha^\mathcal{P}(\xi,\langle{s,\xi}\rangle+u)$ 
at the boundary of $(0,\infty)^n$ 
by using \eqref{equation:taylor201} and \eqref{equation:taylor202}. 
We omit the detail. 
\end{proof}
%
%
\section{Proof of Main Theorem}
\label{section:proofs}
We begin with computing 
$\mathcal{P}_{\alpha}f(s,u)$,  
$\mathcal{Q}_{\alpha}f(s,u)$ 
and  
$\mathcal{R}_{\alpha,\beta}f(s,u)$.  
\begin{lemma}
\label{theorem:compute} 
\quad 
\begin{itemize}
\item[{\rm (i)}] 
For $f(x,y) \in \mathscr{S}_{c,m}^\mathcal{P}(\mathbb{R}^{n+1})$, 
\begin{equation}
\mathcal{P}_{\alpha}f(s,u)
=
\frac{1}{\sqrt{1+\lvert{s}\rvert^2}}
\mathcal{X}F_\alpha^\mathcal{P}
\left(
\frac{(-s,1)}{\sqrt{1+\lvert{s}\rvert^2}},
\frac{u}{\sqrt{1+\lvert{s}\rvert^2}}
\right),
\label{equation:compute1} 
\end{equation}
\begin{equation}
\partial_u^n 
\mathcal{P}_{\alpha}f(s,u)
=
\frac{1}{(1+\lvert{s}\rvert^2)^{(n+1)/2}}
(\partial_t^n \mathcal{X}F_\alpha^\mathcal{P})
\left(
\frac{(-s,1)}{\sqrt{1+\lvert{s}\rvert^2}},
\frac{u}{\sqrt{1+\lvert{s}\rvert^2}}
\right).
\label{equation:compute2} 
\end{equation}
\item[{\rm (ii)}] 
For $f(x,y) \in \mathscr{S}_{c,m}(\mathbb{R}^{n+1})$, 
\begin{equation}
\mathcal{Q}_{\alpha}f(s,u)
=
\frac{1}{\sqrt{1+\lvert{s}\rvert^2}}
\mathcal{X}F_\alpha^\mathcal{Q}
\left(
\frac{(-s,1)}{\sqrt{1+\lvert{s}\rvert^2}},
\frac{u}{\sqrt{1+\lvert{s}\rvert^2}}
\right),
\label{equation:compute5} 
\end{equation}
\begin{equation}
\partial_u^n 
\mathcal{Q}_{\alpha}f(s,u)
=
\frac{1}{(1+\lvert{s}\rvert^2)^{(n+1)/2}}
(\partial_t^n \mathcal{X}F_\alpha^\mathcal{Q})
\left(
\frac{(-s,1)}{\sqrt{1+\lvert{s}\rvert^2}},
\frac{u}{\sqrt{1+\lvert{s}\rvert^2}}
\right).
\label{equation:compute6} 
\end{equation}
\item[{\rm (iii)}] 
For $f(x,y) \in \mathscr{S}_{c,m}^\mathcal{R}(\mathbb{R}^{n+1})$, 
\begin{equation}
\mathcal{R}_{\alpha,\beta}f(s,u)
=
\frac{1}{\sqrt{1+\lvert{s}\rvert^2}}
\mathcal{X}F_{\alpha,\beta}^\mathcal{R}
\left(
\frac{(-s,1)}{\sqrt{1+\lvert{s}\rvert^2}},
\frac{u}{\sqrt{1+\lvert{s}\rvert^2}}
\right),
\label{equation:compute3} 
\end{equation}
\begin{equation}
\partial_u^n 
\mathcal{R}_{\alpha,\beta}f(s,u)
=
\frac{1}{(1+\lvert{s}\rvert^2)^{(n+1)/2}}
(\partial_t^n \mathcal{X}F_{\alpha,\beta}^\mathcal{R})
\left(
\frac{(-s,1)}{\sqrt{1+\lvert{s}\rvert^2}},
\frac{u}{\sqrt{1+\lvert{s}\rvert^2}}
\right).
\label{equation:compute4} 
\end{equation}
\end{itemize}
\end{lemma}
\begin{proof}
Firstly, we prove (i). 
It suffices to prove \eqref{equation:compute1}. 
Using the symmetry \eqref{equation:symmetry1}, we have 
$$
\mathcal{P}_{\alpha}f(s,u)
=
2^n
\int_{(0,\infty)^n}
f\left(x+c,\sum_{i=1}^ns_ix_i^{\alpha_i}+u\right)
dx.
$$
We use the change of variables $x_i=\xi^{1/\alpha_i}$, $i=1,\dotsc,n$. 
Note that 
$$
\frac{dx_i}{d\xi_i}
=
\frac{1}{\alpha_i}
\cdot
\frac{1}{\xi_i^{(\alpha_i-1)/\alpha_i}},
\quad
x_i^{\alpha_i}=\xi_i. 
$$
By using this and \eqref{equation:radonF}, we deduce that 
\begin{align*}
  \mathcal{P}_{\alpha}f(s,u)
& =
  \int_{(0,\infty)^n}
  \frac{2^n f(\xi_1^{1/\alpha_1}+c_1,\dotsc,\xi_n^{1/\alpha_n}+c_n,\langle{s,\xi}\rangle+u)}{\alpha_1 \dotsb \alpha_n \cdot \xi_1^{(\alpha_1-1)/\alpha_1} \dotsb \xi_n^{(\alpha_n-1)/\alpha_n}}
  d\xi
\\
& =
  \int_{(0,\infty)^n}
  F_\alpha^\mathcal{P}(\xi,\langle{s,\xi}\rangle+u)
  d\xi
\\
& =
  \int_{\mathbb{R}^n}
  F_\alpha^\mathcal{P}(\xi,\langle{s,\xi}\rangle+u)
  d\xi
\\
& =
  \frac{1}{\sqrt{1+\lvert{s}\rvert^2}}
  \mathcal{X}F_\alpha^\mathcal{P}
  \left(
  \frac{(-s,1)}{\sqrt{1+\lvert{s}\rvert^2}},
  \frac{u}{\sqrt{1+\lvert{s}\rvert^2}}
  \right),
\end{align*}
which is \eqref{equation:compute1}. 
\par
(ii) can be proved in the same way as (i). We omit the detail. 
\par
Secondly, we prove (iii). 
It suffices to prove \eqref{equation:compute3}. 
Using the symmetry \eqref{equation:symmetry1}, we have 
$$
\mathcal{R}_{\alpha,\beta}f(s,u)
=
2^n
\int_{\substack{x\in(0,\infty)^n \\ \sum s_ix_i^{\alpha_i}+u>0}}
\frac{f\bigl(x+c,(\sum s_ix_i^{\alpha_i}+u)^{1/\beta}\bigr)}{(\sum s_ix_i^{\alpha_i}+u)^{1/\beta}}
dx.
$$
By using the change of variables $x_i=\xi^{1/\alpha_i}$, $i=1,\dotsc,n$, 
and \eqref{equation:radonF}, we deduce that 
\begin{align*}
  \mathcal{R}_{\alpha,\beta}f(s,u)
& =
  \int_{\substack{\xi\in(0,\infty)^n \\ \langle{s,\xi}\rangle+u>0}}
  \frac{2^n f\bigl(\xi_1^{1/\alpha_1}+c_1,\dotsc,\xi_n^{1/\alpha_n}+c_n,(\langle{s,\xi}\rangle+u)^{1/\beta}\bigr)}{\alpha_1 \dotsb \alpha_n \cdot \xi_1^{(\alpha_1-1)/\alpha_1} \dotsb \xi_n^{(\alpha_n-1)/\alpha_n} \cdot (\langle{s,\xi}\rangle+u)^{1/\beta}}
  d\xi
\\
& =
  \int_{(0,\infty)^n}
  F_{\alpha,\beta}^\mathcal{R}(\xi,\langle{s,\xi}\rangle+u)
  d\xi
\\
& =
  \int_{\mathbb{R}^n}
  F_{\alpha,\beta}^\mathcal{R}(\xi,\langle{s,\xi}\rangle+u)
  d\xi
\\
& =
  \frac{1}{\sqrt{1+\lvert{s}\rvert^2}}
  \mathcal{X}F_{\alpha,\beta}^\mathcal{R}
  \left(
  \frac{(-s,1)}{\sqrt{1+\lvert{s}\rvert^2}},
  \frac{u}{\sqrt{1+\lvert{s}\rvert^2}}
  \right),
\end{align*}
which is \eqref{equation:compute2}. 
\end{proof}
Finally we prove Theorem~\ref{theorem:main}.
\begin{proof}[Proof of Theorem~\ref{theorem:main}] 
Firstly, we prove (i). 
Suppose that 
$f(x,y) \in \mathscr{S}_{c,m}^\mathcal{P}(\mathbb{R}^{n+1})$, 
and $m_i\geqq\alpha_i-2$ for all $i=1,\dotsc,n$. 
When $n$ is odd, by using 
the identity \eqref{equation:inverse1}, 
the inversion formula of the standard Radon transform 
Theorem~\ref{theorem:inversionF}, 
the change of the variables \eqref{equation:change111}, 
and the identity \eqref{equation:compute2} in order, 
we deduce that 
\begin{align*}
  f(x,y)
& =
  \frac{1}{2^n}
  \left(
  \prod_{i=1}^n
  \alpha_i
  \lvert{x_i-c_i}\rvert^{\alpha_i-1}
  \right)
  F_{\alpha}^\mathcal{P}
  (\lvert{x_1-c_1}\rvert^{\alpha_1},\dotsc,\lvert{x_n-c_n}\rvert^{\alpha_n},y)
\\
& =
  \frac{4\cdot(-1)^{(n-1)/2}}{(4\pi)^{n+1}}
  \left(
  \prod_{i=1}^n
  \alpha_i
  \lvert{x_i-c_i}\rvert^{\alpha_i-1}
  \right)
\\
& \times
  \int_{\mathbb{S}^n_+}
  \left(
  \operatorname{pv}
  \int_{-\infty}^\infty
  \frac{\partial_t^n \mathcal{X}F_\alpha^\mathcal{P}(\omega,t)}{\sum \omega_i\lvert{x_i-c_i}\rvert^{\alpha_i}+\omega_{n+1}y-t}
  dt
  \right)
  d\omega
\\
& =
  \frac{4\cdot(-1)^{(n-1)/2}}{(4\pi)^{n+1}}
  \left(
  \prod_{i=1}^n
  \alpha_i
  \lvert{x_i-c_i}\rvert^{\alpha_i-1}
  \right)
\\
& \times
  \int_{\mathbb{R}^n}
  \left(
  \operatorname{pv}
  \int_{-\infty}^\infty
  \frac{\sqrt{1+\lvert{s}\rvert^2}}{y-\sum s_i\lvert{x_i-c_i}\rvert^{\alpha_i}-u}
  \right.
\\
& \times
  \left.
  (\partial_t^n \mathcal{X}F_\alpha^\mathcal{P})
  \left(\frac{(-s,1)}{\sqrt{1+\lvert{s}\rvert^2}},\frac{u}{\sqrt{1+\lvert{s}\rvert^2}}\right)
  \cdot
  \frac{1}{(1+\lvert{s}\rvert^2)^{(n+2)/2}}
  du
  \right)
  ds
\\
& =
  \frac{4\cdot(-1)^{(n-1)/2}}{(4\pi)^{n+1}}
  \left(
  \prod_{i=1}^n
  \alpha_i
  \lvert{x_i-c_i}\rvert^{\alpha_i-1}
  \right)
\\
& \times
  \int_{\mathbb{R}^n}
  \left(
  \operatorname{pv}
  \int_{-\infty}^\infty
  \frac{\partial_u^n \mathcal{P}_{\alpha}f(s,u)}{y-\sum s_i\lvert{x_i-c_i}\rvert^{\alpha_i}-u}
  du
  \right)
  ds,
\end{align*}
which is \eqref{equation:inversion1}.
\par
When $n$ is even, by using 
the identity \eqref{equation:inverse1}, 
the inversion formula of the standard Radon transform 
Theorem~\ref{theorem:inversionF}, 
the change of the variables \eqref{equation:change222}, 
and the identity \eqref{equation:compute2} in order, 
we deduce that 
\begin{align*}
  f(x,y)
& =
  \frac{1}{2^n}
  \left(
  \prod_{i=1}^n
  \alpha_i
  \lvert{x_i-c_i}\rvert^{\alpha_i-1}
  \right)
  F_{\alpha}^\mathcal{P}
  (\lvert{x_1-c_1}\rvert^{\alpha_1},\dotsc,\lvert{x_n-c_n}\rvert^{\alpha_n},y)
\\
& =
  \frac{(-1)^{n/2}}{(4\pi)^n}
  \left(
  \prod_{i=1}^n
  \alpha_i
  \lvert{x_i-c_i}\rvert^{\alpha_i-1}
  \right)
\\
& \times
  \int_{\mathbb{S}^n_+}
  \partial_t^n \mathcal{X}F_\alpha^\mathcal{P}
  \left(\omega,\sum_{i=1}^n\omega_i\lvert{x_i-c_i}\rvert^{\alpha_i}+\omega_{n+1}y\right)
  d\omega
\\
& =
  \frac{(-1)^{n/2}}{(4\pi)^n}
  \left(
  \prod_{i=1}^n
  \alpha_i
  \lvert{x_i-c_i}\rvert^{\alpha_i-1}
  \right)
\\
& \times
  \int_{\mathbb{R}^n}
  \frac{1}{(1+\lvert{s}\rvert^2)^{(n+1)/2}}
  \cdot
  (\partial_t^n \mathcal{X}F_\alpha^\mathcal{P})
\left(\frac{(-s,1)}{\sqrt{1+\lvert{s}\rvert^2}},\frac{y-\sum s_i \lvert{x_i-c_i}\rvert^{\alpha_i}}{\sqrt{1+\lvert{s}\rvert^2}}\right)
  ds
\\
& =
  \frac{(-1)^{n/2}}{(4\pi)^n}
  \left(
  \prod_{i=1}^n
  \alpha_i
  \lvert{x_i-c_i}\rvert^{\alpha_i-1}
  \right)
\\
& \times
  \int_{\mathbb{R}^n}
  (\partial_u^n \mathcal{P}_{\alpha}f)
  \left(
  s,
  y-\sum_{i=1}^ns_i \lvert{x_i-c_i}\rvert^{\alpha_i}
  \right)
  ds,
\end{align*}
which is \eqref{equation:inversion2}.
\par
Secondly, we prove (ii). 
Suppose that 
$f(x,y) \in \mathscr{S}_{c,m}(\mathbb{R}^{n+1})$, 
and $m_i\geqq\alpha_i-2$ for all $i=1,\dotsc,n$. 
When $n$ is odd, 
applying 
the inversion formula of the standard Radon transform 
Theorem~\ref{theorem:inversionF}, 
the change of the variables 
\eqref{equation:change111} 
and 
the identity \eqref{equation:compute6} in order,  
we obtain \eqref{equation:inversion5} 
in the exactly same way as \eqref{equation:inversion1}. 
When $n$ is even, 
applying 
the inversion formula of the standard Radon transform 
Theorem~\ref{theorem:inversionF}, 
the change of the variables 
\eqref{equation:change222} 
and 
the identity \eqref{equation:compute6} in order,  
we obtain \eqref{equation:inversion6} 
in the exactly same way as \eqref{equation:inversion2}. 
We omit the detail. 
\par
Finally, we prove (iii). 
Suppose that 
$f(x,y) \in \mathscr{S}_{c,m}^\mathcal{R}(\mathbb{R}^{n+1})$, 
and $m_i\geqq\alpha_i-2$ for all $i=1,\dotsc,n$. 
When $n$ is odd, 
applying 
the inversion formula of the standard Radon transform 
Theorem~\ref{theorem:inversionF}, 
the change of the variables 
\eqref{equation:change111} 
and 
the identity \eqref{equation:compute4} in order,  
we obtain \eqref{equation:inversion3} 
in the exactly same way as \eqref{equation:inversion1}. 
When $n$ is even, 
applying 
the inversion formula of the standard Radon transform 
Theorem~\ref{theorem:inversionF}, 
the change of the variables 
\eqref{equation:change222} 
and 
the identity \eqref{equation:compute4} in order,  
we obtain \eqref{equation:inversion4} 
in the exactly same way as \eqref{equation:inversion2}. 
We omit the detail. 
\end{proof}
%
%
%
%
%
%
\begin{center}
{\sc Acknowledgments} 
\end{center}
\par
The author thank referees for reading the manuscript carefully, 
and giving valuable comments and helpful suggestions. 
%
%

\end{document}